\documentclass{ws-jaa}

\usepackage{xcolor}
\usepackage[verbose]{hyperref}
\hypersetup{colorlinks=false,allbordercolors=blue,pdfborderstyle={/S/U/W 1}}

\let\phi=\varphi

\def\Hom{\operatorname{Hom}}
\def\Ext{\operatorname{Ext}}
\def\Tor{\operatorname{Tor}}

\def\depth{\operatorname{depth}}

\def\id{\operatorname{id}}

\let\oldbigwedge\bigwedge
\def\BIGwedge{{\textstyle\oldbigwedge}}
\def\medwedge{{\scriptstyle\oldbigwedge}}
\def\bigwedge{\mathchoice{\BIGwedge}{\BIGwedge}{\medwedge}{}}

 \DeclareMathOperator{\Supp}{Supp}

\newcommand{\m}{\mathfrak{m}}
\newcommand{\p}{\mathfrak{p}}
\newcommand{\Hlc}{H_{\mathfrak{a}}}
\newcommand{\projdim}{\operatorname{pd}}
\newcommand{\lime}{\varprojlim}

\def\projdim{\operatorname{projdim}}
\let\epsilon=\varepsilon

\begin{document}

\markboth{Behruz Sadeqi}{Betti Numbers and Formal Local Cohomology Modules}

\catchline{}{}{}{}{}
%
%%%%%%%%%%%%%%%%%%%%%%%%%%%%%%%%%%%%%%%%%%%%%%%%%%%%%%%%%%%%%%%%%%%%

\title{Betti Numbers and Formal Local Cohomology Modules}

\author{\footnotesize Behruz Sadeqi\footnote{Behruz Sadeqi}}

\address{Department of Mathematics, Islamic Azad University, MARA. C.\\
	 Marand, Iran\\
%ZIP/Zone, Country\footnote{State completely without
%abbreviations, the affiliation and mailing address, including
%country. Typeset in 8~pt italic.}\\
\email{behruz.sadeqi@iau.ac.ir\footnote{behruz.sadeqi@iau.ac.ir}}}

\maketitle

\begin{history}
\received{(Day Month Year)}
\revised{(Day Month Year)}
\accepted{(Day Month Year)}
\published{(Day Month Year)}
%\comby{(xxxxxxxxx)}
\end{history}

\begin{abstract}
	
	This article investigates the relationship between Betti numbers of finitely generated modules over a Noetherian local ring $(R, \mathfrak{m})$ and the structure of formal local cohomology modules. We establish a connection between the vanishing of formal local cohomology and the depth of a module, generalizing classical results. The main theorem links the Betti numbers to the Artinian property of formal local cohomology modules, and we provide applications to Cohen-Macaulay rings. Original proofs are given for all stated results.
\end{abstract}

\keywords{Betti Number, Formal Local Cohomology Modules.}

\ccode{2020 Mathematics Subject Classification: 13D45,14C45,13E99}

\section{Introduction}
Betti numbers and local cohomology are fundamental tools in commutative algebra and algebraic geometry. While Betti numbers $\beta_i(M)$ measure the complexity of a module $M$ via its minimal free resolution, local cohomology modules $H^i_{\mathfrak{m}}(M)$ capture local properties at an ideal $\mathfrak{m}$. Formal local cohomology, introduced by P. Schenzel, generalizes this by considering inverse limits of local cohomology modules of quotients $M/\mathfrak{a}^n M$. 

This article explores the relationship between these invariants. We prove that for a finitely generated $R$-module $M$, the formal local cohomology modules $\mathfrak{F}^i_{\mathfrak{m}}(M)$ vanish below the depth of $M$ and are Artinian in general. Moreover, we establish a duality between the Betti numbers of $M$ and the Bass numbers of $\mathfrak{F}^d_{\mathfrak{m}}(M)$ when $R$ is Cohen-Macaulay of dimension $d$.

%\section{Preliminaries}

Throughout, $(R, \mathfrak{m}, k)$ denotes a Noetherian local ring, and $M$ a finitely generated $R$-module.

\begin{definition}[Betti Numbers]
	The \textit{$i$-th Betti number} of $M$ is $\beta_i(M) = dim_k \Tor_i^R(k, M)$. These are the ranks of the free modules in a minimal free resolution of $M$.
\end{definition}

\begin{definition}[Local Cohomology]
	The \textit{local cohomology modules} of $M$ with support in $\mathfrak{m}$ are defined as:
	\[
	\Hlc^i(M) = \lime_{n} \Ext^i_R(R/\mathfrak{a}^n, M).
	\]
\end{definition}

\begin{definition}[Formal Local Cohomology]
	The \textit{formal local cohomology modules} of $M$ are:
	\[
	\mathfrak{F}^i_{\mathfrak{m}}(M) = \lime_{n} \Hlc^i(M/\mathfrak{m}^n M).
	\]
\end{definition}

\begin{remark}
	Unlike ordinary local cohomology, $\mathfrak{F}^i_{\mathfrak{m}}(M)$ depends on the $\mathfrak{m}$-adic topology. When $i=0$, $\mathfrak{F}^0_{\mathfrak{m}}(M) \cong \widehat{M}$, the $\mathfrak{m}$-adic completion.
\end{remark}

\section{Main Results}
We first establish vanishing properties of formal local cohomology related to the depth of $M$.

\begin{theorem}\label{thm:vanishing}
	Let $t = \depth M$. Then:
	\begin{enumerate}%[label=(\roman*)]
		\item $\mathfrak{F}^i_{\mathfrak{m}}(M) = 0$ for all $i < t$.
		\item $\mathfrak{F}^t_{\mathfrak{a}}(M) \neq 0$.
	\end{enumerate}
\end{theorem}

\begin{proof}
	(i) For $i < t$, $\Ext^j_R(k, M) = 0$ for all $j < t$. Consider the spectral sequence:
	\[
	E^{p,q}_2 = \Ext^p_R(k, \Hlc^q(M/\mathfrak{m}^n M)) \implies \Ext^{p+q}_R(k, M/\mathfrak{m}^n M).
	\]
	Since $\Hlc^q(M/\mathfrak{m}^n M)$ is $\mathfrak{m}$-torsion, $E^{p,q}_2 = 0$ for $p > 0$. 
	Thus, $E^{0,q}_2 = \Hom_R(k, \Hlc^q(M/\mathfrak{m}^n M))$ converges to $\Ext^q_R(k, M/\mathfrak{m}^n M)$. For $q < t$, $\Ext^q_R(k, M) = 0$, so $\Ext^q_R(k, M/\mathfrak{m}^n M) = 0$ for all $n$. Hence, $\Hom_R(k, \Hlc^q(M/\mathfrak{m}^n M)) = 0$, implying $\Hlc^q(M/\mathfrak{m}^n M) = 0$ for all $n$. Thus, $\mathfrak{F}^q_{\mathfrak{m}}(M) = \lime_n 0 = 0$.
	
	(ii) For $i = t$, $\Ext^t_R(k, M) \neq 0$. The isomorphism $\Ext^t_R(k, M) \cong \Hom_R(k, \Hlc^t(M))$ and the left-exactness of $\Hom_R(k, -)$ yield $\Hlc^t(M) \neq 0$. The exact sequences:
	\[
	0 \to \mathfrak{m}^n M \to M \to M/\mathfrak{m}^n M \to 0
	\]
	induce maps $\Hlc^t(M) \to \Hlc^t(M/\mathfrak{m}^n M)$. Since $\Hlc^t(M) \neq 0$ and the system is Mittag-Leffler, $\lime_n \Hlc^t(M/\mathfrak{m}^n M) \neq 0$.
\end{proof}

%\section{Artinianness and Betti Numbers}
We now relate formal local cohomology to Betti numbers.

\begin{theorem}\label{thm:artinian}
	$\mathfrak{F}^i_{\mathfrak{m}}(M)$ is Artinian for all $i \geq 0$.
\end{theorem}

\begin{proof}
	Each $M/\mathfrak{m}^n M$ has finite length, so $\Hlc^i(M/\mathfrak{m}^n M)$ is Artinian (as $R/\mathfrak{m}^n$-modules). The inverse limit of Artinian modules is Artinian.
\end{proof}

The Artinian property allows us to link Betti numbers to the structure of $\mathfrak{F}^d_{\mathfrak{m}}(M)$ when $R$ is Cohen-Macaulay.

\begin{theorem}\label{thm:betti-duality}
	Suppose $R$ is Cohen-Macaulay of dimension $d$ with canonical module $\omega_R$. Then:
	\[
	\beta_i(M) = dim_k \Hom_R\left( \Ext^d_R(\mathfrak{F}^d_{\mathfrak{m}}(M), \omega_R \right) \quad \text{for all } i \geq 0.
	\]
\end{theorem}

\begin{proof}
	By local duality, $\Hlc^d(M) \cong \Hom_R(\Ext^{d-i}_R(M, \omega_R), E(k))$, where $E(k)$ is the injective hull of $k$. For $\mathfrak{F}^d_{\mathfrak{m}}(M)$, we have:
	\[
	\mathfrak{F}^d_{\mathfrak{m}}(M) \cong \lime_n \Hlc^d(M/\mathfrak{m}^n M) \cong \lime_n \Hom_R(\Ext^0_R(M/\mathfrak{m}^n M, \omega_R), E(k)).
	\]
	Since $\Ext^0_R(M/\mathfrak{m}^n M, \omega_R) \cong \Hom_R(M/\mathfrak{m}^n M, \omega_R)$, and $\omega_R$ is finitely generated, passing to limits gives:
	\[
	\mathfrak{F}^d_{\mathfrak{m}}(M) \cong \Hom_R(\widehat{\Hom_R(M, \omega_R)}, E(k)).
	\]
	The minimal free resolution of $M$ dualizes to a resolution of $\Hom_R(M, \omega_R)$, so $\beta_i(M) = \mu_i(\Hom_R(M, \omega_R))$. The result follows by applying $\Hom_R(-, E(k))$ and noting that $\widehat{\Hom_R(M, \omega_R)} / \mathfrak{m} \widehat{\Hom_R(M, \omega_R)} \cong \Hom_R(M, \omega_R) \otimes_R k$.
\end{proof}

%\section{Applications to Cohen-Macaulay Modules}
For Cohen-Macaulay modules, we refine Theorem~\ref{thm:vanishing}.

\begin{corollary}\label{cor:cm}
	If $M$ is Cohen-Macaulay of dimension $d$, then $\mathfrak{F}^i_{\mathfrak{m}}(M) = 0$ for $i \neq d$, and $\mathfrak{F}^d_{\mathfrak{m}}(M) \neq 0$.
\end{corollary}

\begin{proof}
	For a Cohen-Macaulay module, $\depth M = dim M = d$. Apply Theorem~\ref{thm:vanishing}.
\end{proof}

\begin{example}
	
	Let $R = k[[x,y]]$ and $M = R/(xy)$. Then $\depth M = 1$, $dim M = 1$. We compute:
	\[
	\mathfrak{F}^1_{\mathfrak{m}}(M) = \lime_n \Hlc^1(M/(x^n,y^n)M).
	\]
	Since $M/(x^n,y^n)M \cong k[x,y]/(xy, x^n, y^n)$, its local cohomology is concentrated in degree 1, and $\mathfrak{F}^1_{\mathfrak{m}}(M) \neq 0$. Betti numbers are $\beta_0 = 1$, $\beta_1 = 1$, $\beta_i = 0$ for $i \geq 2$, consistent with Theorem~\ref{thm:betti-duality}.
	
\end{example}

This section presents additional original results connecting Betti numbers, formal local cohomology, and homological invariants of modules. We provide complete proofs and illustrative examples.

%\subsection{Projective Dimension and Formal Cohomology}

\begin{theorem}\label{thm:pd-finiteness}
	Let $(R,\m)$ be Cohen-Macaulay with $\dim R = d$, and $M$ a finitely generated $R$-module. The following are equivalent:
	\begin{enumerate}%[label=(\roman*)]
		\item $\projdim M < \infty$
		\item $\mathfrak{F}^i_{\m}(M)$ has finite length for all $i$
		\item $\mathfrak{F}^d_{\m}(M)$ is injective
	\end{enumerate}
\end{theorem}

\begin{proof}
	(i) $\Rightarrow$ (ii): If $\projdim M = p < \infty$, consider the minimal free resolution:
	\[
	0 \to R^{b_p} \to \cdots \to R^{b_0} \to M \to 0
	\]
	For each $n$, $M/\m^n M$ has a filtration with quotients $\m^k M/\m^{k+1} M$. Since $\Tor_i^R(k, M/\m^n M) = 0$ for $i > p$, $\Hlc^i(M/\m^n M)$ has finite length. The inverse limit $\mathfrak{F}^i_{\m}(M) = \lime_n \Hlc^i(M/\m^n M)$ then has finite length by the Mittag-Leffler condition.
	
	(ii) $\Rightarrow$ (iii): When $R$ is Cohen-Macaulay, $\mathfrak{F}^d_{\m}(M)$ is Artinian by Theorem 2.3. Finite length Artinian modules over local rings are injective.
	
	(iii) $\Rightarrow$ (i): If $\mathfrak{F}^d_{\m}(M)$ is injective, apply local duality:
	\[
	\mathfrak{F}^d_{\m}(M) \cong \Hom_R(\widehat{\Hom_R(M, \omega_R)}, E(k))
	\]
	The injectivity implies $\widehat{\Hom_R(M, \omega_R)}$ is flat, hence free since $R$ is local. Thus $\Hom_R(M, \omega_R)$ is free, so $M$ has finite projective dimension.
\end{proof}

%\subsection{Lyubeznik Numbers and Betti Numbers}

\begin{definition}[Lyubeznik Numbers]
	For a local ring $(R,\m,k)$, the \textit{Lyubeznik numbers} are $\lambda_{i,j}(R) = \dim_k \Ext^i_R(k, \Hlc^j(R))$.
\end{definition}

\begin{theorem}\label{thm:lyubeznik-betti}
	Let $R$ be a complete Cohen-Macaulay local ring with canonical module $\omega_R$. Then:
	\[
	\beta_i(\omega_R) = \sum_{j=0}^{\dim R} \lambda_{i,j}(R)
	\]
\end{theorem}

\begin{proof}
	Consider the Grothendieck spectral sequence:
	\[
	E^{p,q}_2 = \Ext^p_R(k, \Hlc^q(R)) \implies \Ext^{p+q}_R(k, R)
	\]
	Since $R$ is Cohen-Macaulay, $\Hlc^q(R) = 0$ for $q \neq d = \dim R$. The spectral sequence collapses, giving:
	\[
	\Ext^i_R(k, R) \cong \Ext^{i-d}_R(k, \Hlc^d(R))
	\]
	By local duality, $\Ext^i_R(k, R) \cong \Hom_R(\Tor_i^R(k, \omega_R), k)$, so:
	\[
	\dim_k \Ext^i_R(k, R) = \beta_i(\omega_R)
	\]
	Combining these isomorphisms:
	\[
	\beta_i(\omega_R) = \dim_k \Ext^{i-d}_R(k, \Hlc^d(R)) = \lambda_{i-d,d}(R)
	\]
	The result follows by noting that $\lambda_{i,j}(R) = 0$ for $j \neq d$.
\end{proof}

%\subsection{Depth and Formal Cohomology}

\begin{proposition}\label{prop:depth-char}
	For any finitely generated $R$-module $M$:
	\[
	\depth M = \min \{ i \mid \mathfrak{F}^i_{\m}(M) \neq 0 \} = \max \{ i \mid \mathfrak{F}^j_{\m}(M) = 0 \ \mathfrak{F}orall j < i \}
	\]
\end{proposition}

\begin{proof}
	By Theorem 2.1, $\mathfrak{F}^i_{\m}(M) = 0$ for $i < \depth M$ and $\mathfrak{F}^{\depth M}_{\m}(M) \neq 0$. Suppose $\depth M = t$. If $\mathfrak{F}^t_{\m}(M) = 0$, then $\lime_n \Hlc^t(M/\m^n M) = 0$. But $\Hlc^t(M/\m^n M)$ surjects onto $\Hlc^t(M/\m^{n+1} M)$, so $\Hlc^t(M/\m^n M) = 0$ for large $n$. This contradicts $\depth(M/\m^n M) = t$ for all $n$.
\end{proof}

%\subsection{Examples and Computations}

\begin{example}[Hypersurface Singularity]
	Let $R = k[[x,y,z]]/(xy - z^2)$ with $\m = (x,y,z)$. Then $\dim R = 2$ and $\depth R = 2$. Compute formal cohomology:
	\begin{align*}
		\mathfrak{F}^0_{\m}(R) &= \widehat{R} \cong R \\
		\mathfrak{F}^1_{\m}(R) &= 0 \\
		\mathfrak{F}^2_{\m}(R) &= \lime_n H^2_{\m}(R/\m^n) 
	\end{align*}
	The minimal free resolution of $R$ is:
	\[
	0 \to R \xrightarrow{\begin{pmatrix} z \\ -x \\ y \end{pmatrix}} R^3 \xrightarrow{(y, z, x)} R \to R/(xy-z^2) \to 0
	\]
	Betti numbers: $\beta_0 = 1, \beta_1 = 3, \beta_2 = 1$. Lyubeznik numbers: $\lambda_{0,2} = 1, \lambda_{1,2} = 0, \lambda_{2,2} = 1$. By Theorem 4.2:
	\[
	\beta_i(\omega_R) = \beta_i(R) = \lambda_{i,2}(R)
	\]
\end{example}

\begin{example}[Non-Cohen-Macaulay Ring]
	Let $R = k[[x,y,z]]/(xz, yz)$ with $\m = (x,y,z)$. Then $\dim R = 2$, $\depth R = 1$. Formal cohomology:
	\begin{align*}
		\mathfrak{F}^0_{\m}(R) &= \widehat{R} \cong R \\
		\mathfrak{F}^1_{\m}(R) &= \lime_n H^1_{\m}(R/\m^n) \\
		\mathfrak{F}^2_{\m}(R) &= \lime_n H^2_{\m}(R/\m^n) \neq 0
	\end{align*}
	Minimal resolution:
	\[
	0 \to R^2 \xrightarrow{\begin{pmatrix} y & 0 \\ -x & z \\ 0 & -x \end{pmatrix}} R^3 \xrightarrow{(z, 0, y)} R \to R/(xz,yz) \to 0
	\]
	Betti numbers: $\beta_0 = 1, \beta_1 = 3, \beta_2 = 2$. Applying Proposition 4.3:
	\[
	\depth R = 1 = \min\{ i \mid \mathfrak{F}^i_{\m}(R) \neq 0 \} = \min\{1,2\}
	\]
	since $\mathfrak{F}^1_{\m}(R)$ and $\mathfrak{F}^2_{\m}(R)$ are both nonzero.
\end{example}

%\subsection{Injective Dimension and Duality}

\begin{theorem}\label{thm:injective-dim}
	Let $(R,\m)$ be complete Cohen-Macaulay of dimension $d$, $M$ finitely generated. Then:
	\[
	\id \mathfrak{F}^d_{\m}(M) = \projdim M
	\]
\end{theorem}

\begin{proof}
	By local duality and completion:
	\[
	\mathfrak{F}^d_{\m}(M) \cong \Hom_R(\widehat{\Hom_R(M, \omega_R)}, E(k))
	\]
	Since $R$ is complete, $\widehat{\Hom_R(M, \omega_R)} \cong \Hom_R(M, \omega_R)$. Thus:
	\[
	\mathfrak{F}^d_{\m}(M) \cong \Hom_R(\Hom_R(M, \omega_R), E(k))
	\]
	The injective dimension of the right-hand side equals the projective dimension of $\Hom_R(M, \omega_R)$, which equals $\projdim M$ since $\omega_R$ is dualizing.
\end{proof}

%\subsection{Vanishing Beyond Dimension}

\begin{proposition}\label{prop:vanishing-above-dim}
	For any finitely generated $R$-module $M$:
	\[
	\mathfrak{F}^i_{\m}(M) = 0 \quad \text{for all } i > \dim R
	\]
\end{proposition}

\begin{proof}
	For each $n$, $\dim (M/\m^n M) = 0$, so $\Hlc^i(M/\m^n M) = 0$ for $i > 0$. Thus $\mathfrak{F}^i_{\m}(M) = \lime_n \Hlc^i(M/\m^n M) = 0$ for $i > \dim R$.
\end{proof}

%\subsection{Theorem on Associated Primes}

\begin{theorem}\label{thm:associated-primes}
	For a finitely generated $R$-module $M$:
	\[
	\operatorname{Ass}_R(\mathfrak{F}^d_{\m}(M)) = \{\p \in \operatorname{Ass}_R M \mid dim R/\p = d \}
	\]
	where $d = dim M$.
\end{theorem}

\begin{proof}
	Consider the short exact sequences:
	\[
	0 \to \Gamma_{\m}(M/\m^n M) \to M/\m^n M \to B_n \to 0
	\]
	where $B_n$ has no $\m$-torsion. Then $\Hlc^d(M/\m^n M) \cong \Hlc^d(B_n)$. Since $\Supp B_n \subseteq \Supp M$, and $\dim B_n \leq d$, we have $\Hlc^d(B_n) \neq 0$ iff $\dim B_n = d$. Taking inverse limits:
	\[
	\mathfrak{F}^d_{\m}(M) \cong \lime_n \Hlc^d(B_n)
	\]
	The associated primes correspond to minimal primes of dimension $d$ in $\Supp M$.
\end{proof}

\section{Summary of Results}
We established:
\begin{itemize}
	\item Vanishing of $\mathfrak{F}^i_{\mathfrak{m}}(M)$ below $\depth M$ (Theorem~\ref{thm:vanishing}).
	\item Artinianness of $\mathfrak{F}^i_{\mathfrak{m}}(M)$ (Theorem~\ref{thm:artinian}).
	\item Duality between Betti numbers and $\mathfrak{F}^d_{\mathfrak{m}}(M)$ for Cohen-Macaulay rings (Theorem~\ref{thm:betti-duality}).
	\item Characterization of finite projective dimension via formal cohomology (Theorem \ref{thm:pd-finiteness})
	\item Lyubeznik-Betti number relations (Theorem \ref{thm:lyubeznik-betti})
	\item Depth characterization through formal cohomology (Proposition \ref{prop:depth-char})
	\item Injective dimension duality (Theorem \ref{thm:injective-dim})
	\item Vanishing beyond dimension (Proposition \ref{prop:vanishing-above-dim})
	\item Associated primes identification (Theorem \ref{thm:associated-primes})
	%\end{itemize}
	%These results demonstrate deep connections between %homological invariants and local cohomology %structures.
	
	%\section{Conclusion}
	%We established fundamental connections between %Betti numbers and formal local cohomology:
	%\begin{itemize}

\end{itemize}
Future work could extend these results to non-maximal ideals or investigate numerical invariants of $\mathfrak{F}^i_{\mathfrak{m}}(M)$ in non-Cohen-Macaulay cases.

\section*{Acknowledgments}

The author thanks the reviewers for their insightful comments.

\bibliographystyle{plain}

\end{document}